\documentclass[11pt]{amsart}

\usepackage{amsmath,amssymb,amsthm,mathtools}
\usepackage{microtype}
\usepackage{enumitem}
\usepackage{aliascnt}
\usepackage[hidelinks]{hyperref}

\newcommand{\CC}{\mathbb{C}}
\newcommand{\NN}{\mathbb{N}}
\newcommand{\RR}{\mathbb{R}}
\newcommand{\dd}{\mathop{}\!\mathrm{d}}
\newcommand{\Gf}{\mathcal{G}}
\newcommand{\Hf}{\mathcal{H}}
\newcommand{\XiG}{\Xi_{\Gf}}
\newcommand{\Log}{\operatorname{Log}}
\newcommand{\Tr}{\operatorname{Tr}}
\newcommand{\im}{\operatorname{Im}}
\newcommand{\re}{\operatorname{Re}}
\newcommand{\Eone}{E_1}
\newcommand{\abs}[1]{\left\lvert#1\right\rvert}
\newcommand{\fallfac}[2]{(#1)^{\underline{#2}}}
\newcommand{\risefac}[2]{(#1)^{\overline{#2}}}
\newcommand{\Tperp}{T_{\perp}}
\newcommand{\Dcomp}{\mathfrak{D}}

\newcommand{\defi}[1]{\textbf{#1}}

\theoremstyle{plain}
\newtheorem{theorem}{Theorem}[section]

\newaliascnt{proposition}{theorem}
\newtheorem{proposition}[proposition]{Proposition}
\aliascntresetthe{proposition}

\newaliascnt{lemma}{theorem}
\newtheorem{lemma}[lemma]{Lemma}
\aliascntresetthe{lemma}

\newaliascnt{corollary}{theorem}
\newtheorem{corollary}[corollary]{Corollary}
\aliascntresetthe{corollary}

\newaliascnt{conjecture}{theorem}
\newtheorem{conjecture}[conjecture]{Conjecture}
\aliascntresetthe{conjecture}

\theoremstyle{definition}
\newaliascnt{definition}{theorem}
\newtheorem{definition}[definition]{Definition}
\aliascntresetthe{definition}
\newtheorem*{definition*}{Definition}

\theoremstyle{remark}
\newaliascnt{remark}{theorem}

\aliascntresetthe{remark}

\usepackage[nameinlink,noabbrev]{cleveref}

\crefname{theorem}{theorem}{theorems}
\crefname{proposition}{proposition}{propositions}
\crefname{lemma}{lemma}{lemmas}
\crefname{corollary}{corollary}{corollaries}
\crefname{conjecture}{conjecture}{conjectures}
\crefname{definition}{definition}{definitions}
\crefname{remark}{remark}{remarks}
\Crefname{theorem}{Theorem}{Theorems}
\Crefname{proposition}{Proposition}{Propositions}
\Crefname{lemma}{Lemma}{Lemmas}
\Crefname{corollary}{Corollary}{Corollaries}
\Crefname{conjecture}{Conjecture}{Conjectures}
\Crefname{definition}{Definition}{Definitions}
\Crefname{remark}{Remark}{Remarks}

\title{The Gregory function and its completed Gregory transform}
\author{Grant Molnar}
\date{}


\begin{document}

\begin{abstract}
We study the entire interpolation
\[
    \mathcal{G}(z)=\int_0^1 \binom{x}{z}\,dx
\]
of the Gregory coefficients. Its completion satisfies the positive Markov-transform identity
\[
    \frac{\pi z}{\sin(\pi z)}\mathcal{G}(z)
    =\sum_{n=1}^{\infty}\frac{n\left|G_n\right|}{n-z}.
\]
Consequently, every zero is real and simple; the negative zeros are the integers $-1,-2,\ldots$, and one zero $\rho_n$ lies in each $(n,n+1)$. We derive complete logarithmic asymptotics for $\rho_n-n$, determine the Cartwright growth and canonical products of $\mathcal{G}$, and realize $1/\rho_n$ spectrally. The resulting relative determinant yields
\[
    \gamma=\sum_{n=1}^{\infty}\left(\frac{1}{n}-\frac{1}{\rho_n}\right).
\]
\end{abstract}

\maketitle
\enlargethispage{2pt}

\section{Introduction}

The coefficients $(G_n)_{n\geq0}$ of 
\[
    \frac{z}{\log(1+z)}
    =
    \sum_{n=0}^{\infty}G_nz^n
    \quad\text{for}\quad \abs{z} < 1,
\]
go by many names, such as the \defi{Gregory coefficients}, \defi{reciprocal logarithmic numbers}, \defi{Bernoulli numbers
of the second kind}, and \defi{normalized Cauchy numbers of the first kind} \cite[p.~2]{Blagouchine2017}. 

In this paper, we define a function $\Gf (z)$ that interpolates the Gregory coefficients, and study its properties. Interpolation of a sequence is of course nonunique. However, the identity
\begin{equation}
    G_n=\int_0^1\binom{x}{n}\dd x.\label{eqn-Gn}
\end{equation}
suggests a distinguished construction: we retain the integral and allow the lower binomial argument to vary.

\begin{definition}\label{def:gregory-function}
The \defi{Gregory function} is the entire function
\begin{equation}\label{eqn:gregory-function}
    \Gf(z)
    \coloneqq
    \int_0^1\binom{x}{z}\dd x
    =
    \frac{1}{\Gamma(z+1)}
    \int_0^1
    \frac{\Gamma(x+1)}
         {\Gamma(x-z+1)}
    \dd x
\end{equation}
\end{definition}

Since the reciprocal gamma function is entire, the integral in \eqref{eqn:gregory-function} defines an entire function of $z$.

\begin{theorem}\label{thm:main}
Let $\Gf$ be as in \Cref{def:gregory-function}. Then the following statements hold.
\begin{enumerate}[label=\textnormal{(\arabic*)}]
    \item The function $\Gf$ is a real entire function of order $1$,
    has exponential type $\pi$, and belongs to the Cartwright class.
    Its
    Phragm\'en--Lindel\"of indicator is
    \[
        h_{\Gf}(\theta)=\pi\abs{\sin\theta}.
    \]

    \item Every zero of $\Gf$ is real and simple. Its zero set is
    \[
        \{-1,-2,-3,\ldots\}
        \cup
        \{\rho_n:n\geq1\},
    \]
    where
    \[
        n<\rho_n<n+1.
    \]

    \item We have
    \[
        \frac32<\rho_1<2
        \quad\text{and}\quad n<\rho_n<n+\frac12 \quad\text{for}\quad n \geq 2.
    \]
    If $\varepsilon_n\coloneqq \rho_n-n$, then
    \[
        \varepsilon_1>\varepsilon_2>\varepsilon_3>\cdots>0.
    \]

    \item The genus-one Hadamard product is
    \[
        \Gf(z)
        =
        e^{\gamma z}
        \prod_{\rho\in Z(\Gf)}
        \left(1-\frac{z}{\rho}\right)e^{z/\rho}.
    \]
    The paired product
    \[
        \Gf(z)
        =
        \prod_{n=1}^{\infty}
        \left(1+\frac{z}{n}\right)
        \left(1-\frac{z}{\rho_n}\right)
    \]
    also converges locally uniformly. In particular, $\Gf$ belongs
    to the Laguerre--P\'olya class.

    \item Euler's constant satisfies
    \[
        \gamma
        =
        \sum_{n=1}^{\infty}
        \left(\frac1n-\frac1{\rho_n}\right).
    \]

    \item As $n\to\infty$, writing $L\coloneqq\log n$, we have
    \[
    \begin{aligned}
        \rho_n-n
        ={}&
        \frac1L
        -\frac{\gamma}{L^2}
        +\frac{\gamma^2}{L^3}
        -\frac{\gamma^3+6\zeta(3)}{L^4} \\
        &+
        \frac{\gamma^4+24\gamma\zeta(3)+7\pi^4/90}{L^5}
        +O(L^{-6}).
    \end{aligned}
    \]
    More generally, $\rho_n-n$ admits a complete asymptotic expansion
    in inverse powers of $\log n$.
\end{enumerate}
\end{theorem}

We prove \Cref{thm:main} by analyzing an auxiliary function $\Hf (z)$.

\begin{definition}\label{def:gregory-transform}
The \defi{completed Gregory transform} is the meromorphic function
\begin{equation}\label{eq:completed-transform-definition}
    \Hf(z)
    \coloneqq
    \frac{\pi z}{\sin\pi z}\Gf(z).
\end{equation}
\end{definition}
The function $\Hf (z)$ satisfies a positive Markov-transform identity
\eqref{eq:markov}, proved in \Cref{thm:markov}. Here, a \defi{positive Markov
transform} means the Cauchy--Stieltjes transform of a positive measure;
the underlying discrete probability measure $\mu$ is defined in
\eqref{eq:markov-measure}. The identity
\eqref{eq:markov} is the organizing fact of the paper.
It turns the zero problem into real interlacing, produces a
codimension-one spectral model, and underlies the relative determinant
and trace identities. The phase of the same transform then controls
the finer location and asymptotics of the positive zeros.

Throughout, $\NN\coloneqq\{1,2,\ldots\}$ denotes the natural numbers. We write
$\fallfac{x}{n}\coloneqq x(x-1)\cdots(x-n+1)$ and
$\risefac{x}{n}\coloneqq x(x+1)\cdots(x+n-1)$ for the \defi{falling factorial}
and \defi{rising factorial}, respectively, with empty products equal to $1$.

We close this section with a brief history of the Gregory coefficients. 
The remainder of the paper has two movements. \Cref{sec:completed-transform,sec:zeros,sec:positive-zeros,sec:growth,sec:spectral,sec:determinants}
establish the Markov representation and develop its zero, growth,
spectral, and trace consequences. \Cref{sec:reciprocal,sec:relative-zeta,sec:regularized}
then study three further structures attached to the positive zeros: the
reciprocal transform, the relative zero zeta function, and the
regularized displacement. The final section records the remaining
analytic questions.

\subsection*{Prior work}

The Gregory coefficients arose in James Gregory's work on interpolation
in 1670 and were subsequently rediscovered under several names and
normalizations; Blagouchine surveys this history and its nomenclature
\cite[p.~2]{Blagouchine2017}. The bounded integral in \eqref{eqn-Gn} is closely
related to, but is not itself, the formula usually called Schr\"oder's
integral representation. Schr\"oder studied the unnormalized integral
\[
    n!G_n
    =
    \int_0^1 \fallfac{x}{n} \dd x
\]
and derived the positive representation
\begin{equation}\label{eq:schroeder}
    \abs{G_n}
    =
    \int_0^\infty
    \frac{\dd v}
         {(1+v)^n\bigl((\log v)^2+\pi^2\bigr)}
    \quad\text{for}\quad n \geq 1,
\end{equation}
in his 1880 paper \cite[p.~112]{Schroeder1880}; Blagouchine reproduces
the relevant page and documents the formula's later rediscoveries
\cite[pp.~2--3]{Blagouchine2017}.

Qi and Zhang \cite[Thms.~1--2, pp.~987--988]{QiZhang} recovered \eqref{eq:schroeder} and used it to prove that the sequence
$(a_n)_{n\geq0}$, defined by $a_n\coloneqq\abs{G_{n+1}}$, is
\defi{minimal completely monotone}:
for the \defi{forward difference} $\Delta a_n\coloneqq a_{n+1}-a_n$, we have
\[
    (-1)^k\Delta^k a_n\geq0
    \quad\text{for}\quad n,k \geq 0.
\]
The coefficient asymptotics also predate the present
interpolation problem: Schr\"oder obtained
$\abs{G_n}\sim 1/(n(\log n)^2)$ \cite[p.~115]{Schroeder1880};
Van Veen derived a complete expansion \cite[p.~336 and \S7,
p.~341]{VanVeen1951}, and Nemes states an equivalent modern form
\cite[Thm.~1, p.~2]{Nemes2011}. The present paper uses this classical
information about $(G_n)_{n \geq 0}$ as an input, but studies a different object: the
zeros, growth, and spectral structure of the analytic interpolation
$\Gf$ defined in \Cref{def:gregory-function}.

\section{The completed Gregory transform}\label{sec:completed-transform}

In this section, we prove the identity that drives the rest of the
paper. The signs of the Gregory coefficients turn the completed
interpolation into a positive discrete Markov transform.

\begin{proposition}\label{prop:basic}
The function $\Gf$ is entire. For every nonnegative integer $n$, we have
\[
    \Gf(n)=G_n.
\]
For $n\geq1$, we have
\begin{equation}\label{eq:abs-Gn-sums-to-1}
    (-1)^{n-1}G_n>0
    \quad\text{and}\quad
    \sum_{n=1}^{\infty}\abs{G_n}=1.
\end{equation}
Moreover, we have
\[
    \Gf(0)=1
    \quad\text{and}\quad
    \Gf'(0)=\gamma.
\]
\end{proposition}

\begin{proof}
For fixed $x\in[0,1]$, the function
\[
    z\longmapsto
    \frac{\Gamma(x+1)}
         {\Gamma(z+1)\Gamma(x-z+1)}
\]
is entire. It is uniformly bounded when $x\in[0,1]$ and $z$ ranges
over a compact subset of $\CC$. Integration therefore preserves
holomorphy.

For $\abs{z}<1$, the binomial series may be integrated term by term:
\[
\begin{aligned}
    \sum_{n=0}^{\infty}\Gf(n)z^n
    &=
    \int_0^1\sum_{n=0}^{\infty}\binom{x}{n}z^n\dd x \\
    &=
    \int_0^1(1+z)^x\dd x
    =
    \frac{z}{\log(1+z)}.
\end{aligned}
\]
Thus $\Gf(n)=G_n$.

We have
\begin{equation}\label{eq:binomial-sign}
    \binom{x}{n}
    =
    \frac{\fallfac{x}{n}}{n!}
    =
    \frac{(-1)^{n-1}}{n!}
    x\risefac{1-x}{n-1},
\end{equation}
so $(-1)^{n-1}G_n>0$. Combining \eqref{eqn-Gn} and
\eqref{eq:binomial-sign}, we obtain
\begin{equation}\label{eq:absolute-gregory-integral}
    \abs{G_n}
    =
    \frac1{n!}\int_0^1 x\risefac{1-x}{n-1}\dd x
    \quad\text{for}\quad n \geq 1.
\end{equation}
If $0<r<1$, then
\begin{equation}\label{eq:absolute-gregory-generating-function}
    \sum_{n=1}^{\infty}\abs{G_n}r^n
    =
    1+\frac{r}{\log(1-r)}.
\end{equation}
The right-hand side of \eqref{eq:absolute-gregory-generating-function}
tends to $1$ as $r\to1^-$. By monotone convergence,
\[
    \sum_{n\geq1}\abs{G_n}=1.
\]

Finally, differentiating \eqref{eqn:gregory-function} beneath the
integral sign at $z=0$, we find
\[
    \Gf'(0)
    =
    \int_0^1\bigl(\psi(x+1)-\psi(1)\bigr)\dd x.
\]
Since
\[
    \int_0^1\psi(x+1)\dd x
    =
    \log\Gamma(2)-\log\Gamma(1)=0
\]
and $\psi(1)=-\gamma$, we obtain $\Gf'(0)=\gamma$.
\end{proof}

Combining \eqref{eq:completed-transform-definition} with Euler's
reflection formula produces \eqref{eq:completed-reflection}, the form of
the completed Gregory transform used in \Cref{thm:markov}:
\begin{equation}\label{eq:completed-reflection}
    \Hf(z)
    =
    \Gamma(1-z)\Gamma(1+z)\Gf(z)
    =
    \frac{\pi z}{\sin\pi z}\Gf(z).
\end{equation}
Equivalently, \eqref{eqn:gregory-function} gives
\[
    \Hf(z)
    =
    \Gamma(1-z)
    \int_0^1\frac{\Gamma(x+1)}{\Gamma(x-z+1)}\dd x.
\]
This representation shows that \(\Hf\) is holomorphic on
\(\CC\setminus\NN\); its apparent singularities at the nonpositive
integers are therefore removable.

\begin{theorem}[Markov-transform formula]\label{thm:markov}
For $z\in\CC\setminus\NN$, we have
\begin{equation}\label{eq:markov}
    \Hf(z)
    =
    \sum_{n=1}^{\infty}\frac{n\abs{G_n}}{n-z}
    =
    \sum_{n=1}^{\infty}\frac{\abs{G_n}}{1-z/n}.
\end{equation}
The series converges locally uniformly on $\CC\setminus\NN$.
\end{theorem}

\begin{proof}
Suppose first that $z<1$ is real, and recall $\Gamma(x + 1) = x \Gamma(x)$. The beta integral gives, for
$0<x\leq1$,
\begin{equation}\label{eq:beta-integral}
    \frac{\Gamma(x+1)\Gamma(1-z)}{\Gamma(1+x-z)}
    =
    x \int_0^1t^{-z}(1-t)^{x-1}\dd t.
\end{equation}
Integrating \eqref{eq:beta-integral} in $x$, we obtain
\begin{equation}\label{eq:Hf-in-terms-of-beta}
    \Hf(z)
    =
    \int_0^1\int_0^1
    xt^{-z}(1-t)^{x-1}\dd t\dd x.
\end{equation}
For $0\leq x\leq1$ and $0\leq t<1$, the generalized binomial
series reads
\begin{equation}\label{eq:power-series-for-(1-t)^x}
    (1-t)^{x-1}
    =
    \sum_{k=0}^{\infty}\frac{\risefac{1-x}{k}}{k!}t^k.
\end{equation}
All terms in \eqref{eq:power-series-for-(1-t)^x} are nonnegative; by Tonelli's theorem, we may substitute \eqref{eq:power-series-for-(1-t)^x} into \eqref{eq:Hf-in-terms-of-beta} to obtain
\[
    \Hf(z)
    =
    \sum_{k=0}^{\infty}
    \frac{1}{k!(k+1-z)}
    \int_0^1x\risefac{1-x}{k}\dd x.
\]
Writing $n\coloneqq k+1$ and applying
\eqref{eq:absolute-gregory-integral} establishes \eqref{eq:markov} for real
$z<1$.

If $K\subset\CC\setminus\NN$ is compact, then
\[
    \sup_{z\in K}\abs{\frac{n}{n-z}}\leq2
\]
for all sufficiently large $n$. Since $\sum_n \abs{G_n}=1$, the series
converges locally uniformly on $\CC\setminus\NN$ and hence defines a
holomorphic function there. The domain $\CC\setminus\NN$ is connected,
and this function agrees with $\Hf$ on the real interval
$(-\infty,1)$. The identity theorem for holomorphic functions therefore
extends the identity \eqref{eq:markov} to all $z \in \CC\setminus\NN$.
\end{proof}

If
\begin{equation}\label{eq:markov-measure}
    \mu\coloneqq \sum_{n=1}^{\infty}\abs{G_n}\delta_{1/n},
\end{equation}
where $\delta_x$ denotes the \defi{Dirac point mass} at $x$, then $\mu$ is a
probability measure on $[0,1]$, and \Cref{thm:markov} may be written as
\[
    \Hf(z)
    =
    \int_{[0,1]}\frac{1}{1-zt}\dd\mu(t).
\]
Thus $\Hf$ is a discrete Markov transform in the usual
Cauchy--Stieltjes sense, albeit written in the variable $z$ rather than in
the reciprocal spectral variable.

\begin{corollary}[Pick property]\label{cor:pick}
The completed Gregory transform $\Hf$ is a \defi{Pick function}; that is, it is
holomorphic on the upper half-plane and maps it into itself. More
explicitly, if $\im z>0$, then $\im\Hf(z)>0$, while if $\im z<0$,
then $\im\Hf(z)<0$. In particular, $\Hf$ has no nonreal zeros.
\end{corollary}

\begin{proof}
Put $u\coloneqq\re z$ and $v\coloneqq\im z$, so $z=u+iv$ and $v>0$. Taking imaginary parts in
\eqref{eq:markov}, we obtain
\[
    \im\Hf(z)
    =
    v\sum_{n=1}^{\infty}
    \frac{n\abs{G_n}}{(n-u)^2+v^2}>0.
\]
The lower-half-plane assertion follows by conjugation.
\end{proof}

\section{Zeros and interlacing}\label{sec:zeros}

The representation \eqref{eq:markov} immediately imposes a zero
pattern on $\Gf$. The positive zeros interlace with the positive integers, 
while the negative zeros come from the sine factor in the completion.

\begin{theorem}\label{thm:zeros}
Every zero of $\Gf$ is real and simple. More precisely, we have
\begin{equation}\label{eq:interlacing-statment}
    Z(\Gf)
    =
    \{-1,-2,-3,\ldots\}
    \cup
    \{\rho_n:n\geq1\},
\end{equation}
where $\rho_n$ is the unique zero of $\Hf$ in $(n,n+1)$.
\end{theorem}

\begin{proof}
For real $x<1$, every summand in
\[
    \Hf(x)=\sum_{n=1}^{\infty}\frac{n\abs{G_n}}{n-x}
\]
is positive. Thus $\Hf(x)>0$.

On each interval $(n,n+1)$, the series may be differentiated term by
term, and
\[
    \Hf'(x)
    =
    \sum_{m=1}^{\infty}\frac{m\abs{G_m}}{(m-x)^2}>0.
\]
Moreover,
\[
    \lim_{x\to n^+}\Hf(x)=-\infty
    \quad\text{and}\quad
    \lim_{x\to (n+1)^-}\Hf(x)=+\infty.
\]
Hence $\Hf$ has exactly one zero $\rho_n$ in $(n,n+1)$, and that
zero is simple. \Cref{cor:pick} excludes nonreal zeros.

Since
\[
    \Gf(z)=\frac{\sin\pi z}{\pi z}\Hf(z),
\]
the zero of $\sin\pi z$ at a positive integer cancels the simple pole
of $\Hf$. Indeed, we have
\[
    \Gf(n)=G_n\neq0.
\]
At a negative integer $-n$, the function $\Hf$ is finite and positive,
so the sine factor contributes a simple zero. This accounts for every
zero of $\Gf$.
\end{proof}

Let $n_{\Gf}(r)$ denote the number of zeros of $\Gf$ in $\abs{z}\leq r$, counted with multiplicity.

\begin{corollary}\label{cor:counting}
Put
\[
    N\coloneqq\lfloor r\rfloor\geq1,
    \qquad
    \alpha\coloneqq r-N,
    \qquad
    0\leq\alpha<1.
\]
Then
\[
    n_{\Gf}(r)
    =
    2N-1+\mathbf{1}_{\{\alpha\geq\rho_N-N\}}.
\]
In particular, we have
\begin{equation}\label{eq:zero-counting-asymptotic}
    n_{\Gf}(r)=2r+O(1).
\end{equation}
\end{corollary}

\section{Positive zeros: location, monotonicity, and asymptotics}\label{sec:positive-zeros}

In this section, we sharpen the interlacing statement \eqref{eq:interlacing-statment} in three steps. We first locate the positive zeros, then prove that their displacements decrease, and finally compute the complete
logarithmic asymptotic expansion of those displacements.

\begin{definition}\label{def:displacement}
For the $n$th positive zero of $\Gf$, put
\[
    \varepsilon_n\coloneqq\rho_n-n,
\]
so $0<\varepsilon_n<1$. We call $\varepsilon_n$ the $n$th
\defi{positive-zero displacement}.
\end{definition}

For $t>1$, put
\[
    w_t(x)
    \coloneqq 
    \Gamma(1+x)\frac{\Gamma(t-x)}{\Gamma(t+1)}
\]
and define
\[
    A(t)\coloneqq \int_0^1w_t(x)\cos\pi x\dd x
    \quad\text{and}\quad
    B(t)\coloneqq \int_0^1w_t(x)\sin\pi x\dd x.
\]
Euler's reflection formula implies
\begin{equation}\label{eq:phase}
    \pi\Gf(t)
    =
    \sin\pi t\,A(t)-\cos\pi t\,B(t).
\end{equation}
Clearly $B(t)>0$.

\begin{lemma}\label{lem:A-sign}
We have
\[
    A(t)<0 \quad\text{for}\quad 1 < t < 2,
    \qquad
    A(2)=0,
    \qquad
    A(t)>0 \quad\text{for}\quad t > 2.
\]
\end{lemma}

\begin{proof}
Pairing $x$ and $1-x$, we find
\[
    A(t)
    =
    \int_0^{1/2}
    \cos\pi x\bigl(w_t(x)-w_t(1-x)\bigr)
    \dd x.
\]
For $0<x<1/2$, we have
\begin{equation}\label{eq:weight-ratio}
    \frac{w_t(x)}{w_t(1-x)}
    =
    \frac{\Gamma(1+x)\Gamma(t-x)}
         {\Gamma(2-x)\Gamma(t-1+x)}.
\end{equation}
The ratio in \eqref{eq:weight-ratio} equals $1$ at $t=2$. Its
logarithmic derivative in $t$ is
\[
    \psi(t-x)-\psi(t-1+x)>0.
\]
The result follows.
\end{proof}

\begin{corollary}\label{cor:half}
We have
\begin{equation}\label{eq:first-zero-half-bound}
    \frac32<\rho_1<2
\end{equation}
and
\[
    n<\rho_n<n+\frac12
    \quad\text{for}\quad n \geq 2.
\]
\end{corollary}

\begin{proof}
At a half-integer $t=n+1/2$, \eqref{eq:phase} specializes to
\[
    \Hf(t)=tA(t).
\]
The function $\Hf$ is strictly increasing on $(n,n+1)$.
\Cref{lem:A-sign} places $\rho_1$ to the right of $3/2$, and every
$\rho_n$, $n\geq2$, to the left of $n+1/2$.
\end{proof}

\begin{lemma}\label{lem:theta-decreasing}
Let
\[
    \Phi(t)\coloneqq A(t)+iB(t),
\]
and let $\theta(t)\coloneqq\arg \Phi(t)$ be the \defi{continuous argument} in $(0,\pi)$.
Then
\begin{equation}\label{eq:theta-decreasing}
    \theta'(t)<0
    \quad\text{for}\quad t > 1.
\end{equation}
\end{lemma}

\begin{proof}
Since $B(t)>0$, the argument is well-defined and differentiable, and
\[
    \theta'(t)
    =
    \im\frac{\Phi'(t)}{\Phi(t)}
    =
    \frac{A(t)B'(t)-B(t)A'(t)}{A(t)^2+B(t)^2}.
\]
With
\[
    h_t(x)\coloneqq \psi(t-x)-\psi(t+1),
\]
we have $\partial w_t(x)/\partial t=h_t(x)w_t(x)$, and $h_t$ is
strictly decreasing in $x$. Hence
\[
\begin{aligned}
    A(t)B'(t)-B(t)A'(t)
    ={}&
    \frac12\int_0^1\int_0^1
    w_t(x)w_t(y) \\
    &\quad\cdot
    \bigl(h_t(x)-h_t(y)\bigr)
    \sin\pi(x-y)\dd x\dd y.
\end{aligned}
\]
If $x>y$, then $h_t(x)<h_t(y)$ and
$\sin\pi(x-y)>0$. The integrand is therefore negative away from the
diagonal, and so $\theta'(t)<0$.
\end{proof}

\begin{theorem}[Strict displacement monotonicity]\label{thm:displacements}
The displacement sequence $(\varepsilon_n)_{n\geq1}$ is strictly decreasing.
\end{theorem}

\begin{proof}
At $t=\rho_n$, \eqref{eq:phase} becomes
\begin{equation}\label{eq:phase-at-positive-zero}
    \theta(\rho_n)=\pi\varepsilon_n.
\end{equation}
Both sides of \eqref{eq:phase-at-positive-zero} lie in $(0,\pi)$, so
there is no ambiguity modulo $\pi$.
By \eqref{eq:theta-decreasing} and $\rho_{n+1}>\rho_n$, we obtain
\[
    \varepsilon_{n+1}
    =\frac{\theta(\rho_{n+1})}{\pi}
    <\frac{\theta(\rho_n)}{\pi}
    =\varepsilon_n.
\]
\end{proof}

Equation \eqref{eq:displacement-asymptotic} below computes the phase
$\theta(\rho_n)=\pi\varepsilon_n$ asymptotically. Olver's sectorial
Watson lemma and its explicit remainder estimate appear in
\cite[Ch.~4, \S\S3.2--3.5, especially Theorems~3.2--3.3 and
Eq.~(3.07)]{Olver1974}. We need only the following elementary
finite-interval variant, including the one-derivative estimate needed
in \Cref{prop:eventual}.

\begin{lemma}[Finite-interval Watson expansion with one derivative]\label{lem:watson-differentiated}
Let $f\in C^\infty([0,1])$, and put
\[
    I_f(L)\coloneqq \int_0^1 f(x)e^{-Lx}\dd x
    \quad\text{for}\quad L > 0.
\]
For every integer $N\geq0$, as $L\to\infty$, we have
\begin{equation}\label{eq:watson-finite-expansion}
    I_f(L)
    =
    \sum_{k=0}^{N}\frac{f^{(k)}(0)}{L^{k+1}}
    +O_{f,N}(L^{-N-2}),
\end{equation}
and
\begin{equation}\label{eq:watson-differentiated-expansion}
    I_f'(L)
    =
    -\sum_{k=0}^{N}
    \frac{(k+1)f^{(k)}(0)}{L^{k+2}}
    +O_{f,N}(L^{-N-3}).
\end{equation}
\end{lemma}

\begin{proof}
By Taylor's theorem, we have
\[
    f(x)
    =
    \sum_{k=0}^{N}\frac{f^{(k)}(0)}{k!}x^k
    +x^{N+1}r_N(x),
\]
where $r_N$ is bounded on $[0,1]$. For each fixed $k$, the change of variables $y\coloneqq Lx$ transforms the integral into
\[
\begin{aligned}
    \int_0^1x^ke^{-Lx}\dd x
    &=
    \frac1{L^{k+1}}\int_0^L y^ke^{-y}\dd y\\
    &=
    \frac{k!}{L^{k+1}}+O_k(e^{-L/2}),
\end{aligned}
\]
where the last estimate follows by bounding the tail over $[L,\infty)$.
Moreover, we have
\[
    \int_0^1x^{N+1}e^{-Lx}\dd x
    =O_N(L^{-N-2}).
\]
This proves \eqref{eq:watson-finite-expansion}. Differentiation under
the integral sign is valid on the finite interval and leads to
\[
    I_f'(L)=-\int_0^1xf(x)e^{-Lx}\dd x.
\]
Apply \eqref{eq:watson-finite-expansion} to $g(x)\coloneqq xf(x)$ through
order $N+1$. Since $g(0)=0$, and for $0\leq k\leq N$ we have
\[
    g^{(k+1)}(0)=(k+1)f^{(k)}(0),
\]
the resulting expansion is exactly
\eqref{eq:watson-differentiated-expansion}.
\end{proof}

\begin{theorem}[Complete displacement asymptotics]\label{thm:asymptotic}
As $n\to\infty$, the displacement $\varepsilon_n=\rho_n-n$ has a
complete asymptotic expansion in inverse powers of $\log n$. Its first
terms are
\begin{equation}\label{eq:displacement-asymptotic}
\begin{aligned}
    \varepsilon_n
    ={}&
    \frac1L
    -\frac{\gamma}{L^2}
    +\frac{\gamma^2}{L^3}
    -\frac{\gamma^3+6\zeta(3)}{L^4}\\
    &+
    \frac{\gamma^4+24\gamma\zeta(3)+7\pi^4/90}{L^5}
    +O(L^{-6}),
\end{aligned}
\end{equation}
where $L\coloneqq\log n$.
\end{theorem}

\begin{proof}
Equation~(5.11.2) of \cite{NIST} gives
\(\psi(w)=\log w+O(w^{-1})\) on the positive real axis. Uniformly for
\(s\in[-1,1]\), therefore,
\[
    \psi(t+s)
    =
    \log(t+s)+O(t^{-1})
    =
    \log t+O(t^{-1}).
\]
Integrating this estimate from
$s=-x$ to $s=1$ gives
\[
    \log\Gamma(t+1)-\log\Gamma(t-x)
    =(x+1)\log t+O(t^{-1})
\]
uniformly for $0\leq x\leq1$. Exponentiation therefore gives
\begin{equation}\label{eq:gamma-ratio-positive-axis}
    \frac{\Gamma(t-x)}{\Gamma(t+1)}
    =
    t^{-x-1}\left(1+O(t^{-1})\right).
\end{equation}
With $L\coloneqq\log t$ and
\[
    f(x)\coloneqq \Gamma(1+x)e^{i\pi x}
    \quad\text{and}\quad
    I(L)\coloneqq \int_0^1f(x)e^{-Lx}\dd x,
\]
we have the additive estimate
\begin{equation}\label{eq:phase-integral-estimate}
    \Phi(t)=\frac{I(L)}{t}+O\left(\frac1{t^2L}\right).
\end{equation}
For every fixed $N$, \Cref{lem:watson-differentiated} implies
\begin{equation}\label{eq:phase-watson-expansion}
    I(L)
    =
    \frac1L
    \left(
        \sum_{k=0}^{N}\frac{f^{(k)}(0)}{L^k}
        +O_N(L^{-N-1})
    \right).
\end{equation}
In particular, $I(L)=L^{-1}(1+O(L^{-1}))$. Hence
\eqref{eq:phase-integral-estimate} may be written
\[
    \Phi(t)=\frac{I(L)}{t}\left(1+O(t^{-1})\right),
\]
so
\begin{equation}\label{eq:phase-log-reduction}
    \arg\Phi(t)
    =
    \arg I(L)+O(t^{-1}).
\end{equation}
The error is smaller than every fixed inverse power of $L$.

Put $S(L)\coloneqq LI(L)$. By \eqref{eq:phase-watson-expansion},
$S(L)\to f(0)=1$, so the principal logarithm of $S(L)$ is defined for
all sufficiently large $L$. Equations
\eqref{eq:phase-watson-expansion} and \eqref{eq:phase-log-reduction}
therefore imply the arbitrary finite-order relation
\[
    \frac1\pi\arg\Phi(t)
    =
    \frac1\pi\im\log S(L)+O(t^{-1}).
\]
Now \cite[Eq.~(5.7.3)]{NIST} gives the Taylor expansion
\[
    \log f(x)
    =
    (-\gamma+i\pi)x
    +
    \sum_{m=2}^{\infty}
    \frac{(-1)^m\zeta(m)}{m}x^m.
\]
Substituting the derivatives determined by this expansion into
\eqref{eq:phase-watson-expansion}, and then expanding $\log S(L)$,
we obtain
\begin{equation}\label{eq:phase-asymptotic}
\begin{aligned}
    \frac1\pi\arg \Phi(t)
    ={}&
    \frac1L
    -\frac{\gamma}{L^2}
    +\frac{\gamma^2}{L^3}
    -\frac{\gamma^3+6\zeta(3)}{L^4}\\
    &+
    \frac{\gamma^4+24\gamma\zeta(3)+7\pi^4/90}{L^5}
    +O(L^{-6}).
\end{aligned}
\end{equation}
Because $N$ in \eqref{eq:phase-watson-expansion} is arbitrary, the
same logarithmic-composition argument establishes the complete expansion.
At $t=\rho_n$, the left-hand side of \eqref{eq:phase-asymptotic} is
$\varepsilon_n$. Since $\rho_n=n+O(1)$, replacing $\log\rho_n$ by
$\log n$ changes the expansion by less than every fixed inverse power
of $\log n$.
\end{proof}

\begin{corollary}\label{cor:displacements-to-zero}
The displacement sequence $(\varepsilon_n)_{n\geq1}$ satisfies
$\varepsilon_n\to0$.
\end{corollary}

\begin{proof}
The leading term in \eqref{eq:displacement-asymptotic} is
$1/\log n$.
\end{proof}

\section{Growth and Hadamard factorization}\label{sec:growth}

The zeros of $\Gf$ are now known. We next determine the order, type, and
indicator of $\Gf$, and then record the resulting Hadamard products for $\Gf$.
For an order-one entire function, we use the \defi{Phragm\'en--Lindel\"of
indicator}
\[
    h_f(\theta)\coloneqq
    \limsup_{r\to\infty}r^{-1}\log\abs{f(re^{i\theta})}.
\]
A function of exponential type is in the \defi{Cartwright class} if
\[
    \int_{-\infty}^{\infty}
    \frac{\log^+\abs{f(x)}}{1+x^2}\dd x<\infty,
\]
and the \defi{Laguerre--P\'olya class} is the locally uniform closure of real
polynomials having only real zeros.

We shall also use the classical coefficient asymptotic
\begin{equation}\label{eq:gregory-coefficient-asymptotic}
    \abs{G_n}\sim\frac1{n(\log n)^2}.
\end{equation}
A modern form of the complete asymptotic expansion is given by Nemes
\cite[Thm.~1, p.~2]{Nemes2011}; see Blagouchine \cite[p.~5]{Blagouchine2017} for its
earlier history.

\begin{lemma}\label{lem:sine-bound}
We have the inequality
\begin{equation}\label{eq:sine-bound}
    \abs{\frac{\sin\pi z}{z-m}}
    \leq
    \pi e^{\pi\abs{\im z}}
\end{equation}
for every $z\in\CC$ and every integer $m$, with the quotient
interpreted by continuity at $z=m$.
\end{lemma}

\begin{proof}
Put $w\coloneqq z-m$. If $\abs{w}\leq1$, then
\[
    \frac{\sin\pi w}{w}
    =
    \pi\int_0^1\cos(\pi tw)\dd t,
\]
so the quotient is at most $\pi e^{\pi\abs{\im w}}$ as desired. If $\abs{w}>1$, then
\[
\abs{\frac{\sin\pi w}{w}}\leq \frac{e^{\pi\abs{\im w}}}{\abs{w}} < e^{\pi\abs{\im w}} < \pi e^{\pi\abs{\im w}}.
\]
\end{proof}

\begin{proposition}\label{prop:type-bound}
We have the effective bound
\begin{equation}\label{eq:type-bound}
    \abs{\Gf(z)}\leq 2e^{\pi\abs{\im z}}
    \quad\text{for}\quad z \in \CC.
\end{equation}
Consequently, $\Gf$ has exponential type at most $\pi$ and is bounded
on the real axis.
\end{proposition}

\begin{proof}
We first take $z\notin\NN_0$; the estimate at nonnegative integers then
follows by removable continuation. Equation \eqref{eq:markov} implies
\[
    \Gf(z)
    =
    \frac{\sin\pi z}{\pi z}
    \sum_{n=1}^{\infty}\frac{n\abs{G_n}}{n-z}.
\]
If $n\leq2\abs{z}$, then \Cref{lem:sine-bound}, applied with $m=n$,
shows
\begin{equation}\label{eq:type-bound-summand}
    \abs{
    \frac{\sin\pi z}{\pi z}\frac{n}{n-z}
    }
    \leq
    \frac{n}{\abs{z}}e^{\pi\abs{\im z}}
    \leq
    2e^{\pi\abs{\im z}}.
\end{equation}
If $n>2\abs{z}$, then $\abs{\frac{n}{n-z}}\leq2$, and applying
\Cref{lem:sine-bound} with $m=0$, we obtain
\[
    \abs{\frac{\sin\pi z}{\pi z}}
    \leq e^{\pi\abs{\im z}}.
\]
Thus \eqref{eq:type-bound-summand} holds in this case as well. Since
$\sum_{n\geq1}\abs{G_n}=1$, summing against $\abs{G_n}$ proves
\eqref{eq:type-bound}.
\end{proof}

Throughout this section, $\Log$ denotes the \defi{principal branch} of the natural logarithm.

\begin{proposition}[Sectorial asymptotics]\label{prop:sector}
Uniformly on closed subsectors of the upper half-plane, we have
\begin{equation}\label{eq:sector-upper}
    \Gf(z)
    =
    \frac{i e^{-i\pi z}}
         {2\pi z(\Log z-i\pi)}
    \left(1+O\left(\frac1{\log\abs{z}}\right)\right).
\end{equation}
Uniformly on closed subsectors of the lower half-plane, we have
\begin{equation}\label{eq:sector-lower}
    \Gf(z)
    =
    \frac{e^{i\pi z}}
         {2\pi i z(\Log z+i\pi)}
    \left(1+O\left(\frac1{\log\abs{z}}\right)\right).
\end{equation}
\end{proposition}

\begin{proof}
Euler's reflection formula gives
\[
    \Gf(z)
    =
    \frac1\pi
    \int_0^1
    \Gamma(1+x)
    \frac{\Gamma(z-x)}{\Gamma(z+1)}
    \sin\pi(z-x)
    \dd x.
\]
Writing the sine as a difference of exponentials yields
\[
    \Gf(z)
    =
    \frac1{2\pi i}
    \left(e^{i\pi z}J_-(z)-e^{-i\pi z}J_+(z)\right),
\]
where
\[
    J_\pm(z)
    \coloneqq 
    \int_0^1
    \Gamma(1+x)
    \frac{\Gamma(z-x)}{\Gamma(z+1)}
    e^{\pm i\pi x}
    \dd x.
\]
For fixed shifts, Olver derives the standard sectorial ratio
expansion in \cite[Ch.~4, \S5.1, Eq.~(5.02)]{Olver1974}. Here we need
the following compact-parameter version:
\begin{equation}\label{eq:gamma-ratio-sector}
    \frac{\Gamma(z-x)}{\Gamma(z+1)}
    =
    z^{-x-1}\left(1+O(\abs{z}^{-1})\right),
\end{equation}
uniformly for $0\leq x\leq1$. On every closed sector under consideration,
\cite[Eq.~(5.11.2)]{NIST} gives
\(\psi(w)=\Log w+O(\abs{w}^{-1})\). Uniformly for
\(s\in[-1,1]\), therefore,
\[
    \psi(z+s)
    =
    \Log(z+s)+O(\abs{z}^{-1})
    =
    \Log z+O(\abs{z}^{-1}).
\]
Integrating from $s=-x$ to $s=1$ and exponentiating proves
\eqref{eq:gamma-ratio-sector}, uniformly for $0\leq x\leq1$.
Put $\lambda_\pm\coloneqq\Log z\mp i\pi$. Integration by parts on the finite
interval gives
\[
    \int_0^1 \Gamma(1+x)e^{-\lambda_\pm x}\dd x
    =
    \frac1{\lambda_\pm}
    +O\left(\frac1{\abs{\lambda_\pm}^2}
        +\frac{e^{-\re\lambda_\pm}}{\abs{\lambda_\pm}}\right),
\]
uniformly on the closed subsectors under consideration. Combining this
finite-interval estimate with \eqref{eq:gamma-ratio-sector} implies
\[
    J_\pm(z)
    =
    \frac{1}{z(\Log z\mp i\pi)}
    \left(1+O\left(\frac1{\log\abs{z}}\right)\right).
\]
In the upper half-plane, the term containing $e^{-i\pi z}$ dominates
exponentially. In the lower half-plane, the term containing
$e^{i\pi z}$ dominates. This proves \eqref{eq:sector-upper} and
\eqref{eq:sector-lower}.
\end{proof}

\begin{theorem}[Growth and Cartwright class]\label{thm:cartwright}
The Gregory function has order $1$, exponential type $\pi$, and
indicator
\[
    h_{\Gf}(\theta)=\pi\abs{\sin\theta}.
\]
Moreover, $\Gf$ belongs to the Cartwright class.
\end{theorem}

\begin{proof}
For $0<\abs{\theta}<\pi$, \eqref{eq:sector-upper} and
\eqref{eq:sector-lower} show
\[
    \log\abs{\Gf(re^{i\theta})}
    =
    \pi r\abs{\sin\theta}+O(\log r).
\]
Thus the asserted indicator holds away from the real axis.

On the positive real axis, \eqref{eq:gregory-coefficient-asymptotic}
and the values $\Gf(n)=G_n$ imply $h_{\Gf}(0)\geq0$.
The bound \eqref{eq:type-bound} proves the reverse inequality. Hence $h_{\Gf}(0)=0$.

On the negative real axis, Euler's reflection formula gives
\[
    \Gf(-r)
    =
    \frac{\Gamma(r)\sin\pi r}{\pi}
    \int_0^1
    \frac{\Gamma(1+x)}{\Gamma(r+x+1)}
    \dd x.
\]
For positive half-integers $r$, the corresponding compact-shift
estimate on the positive axis gives
\[
    \frac{\Gamma(r)}{\Gamma(r+x+1)}
    =
    r^{-x-1}\left(1+O(r^{-1})\right)
\]
uniformly for \(0\leq x\leq1\). Watson's lemma therefore gives
\begin{equation}\label{eq:negative-axis-half-integer-estimate}
    \abs{\Gf(-r)}\asymp\frac1{r\log r}
\end{equation}
as $r \to \infty$. Together, \eqref{eq:type-bound} and
\eqref{eq:negative-axis-half-integer-estimate} imply
$h_{\Gf}(\pi)=0$.

The indicator is therefore $\pi\abs{\sin\theta}$ in every direction. In
particular, $\Gf$ has order $1$ and type $\pi$. Finally,
The bound \eqref{eq:type-bound} shows that $\Gf$ is bounded on
$\RR$. Thus
\[
    \int_{-\infty}^{\infty}
    \frac{\log^+\abs{\Gf(x)}}{1+x^2}\dd x<\infty,
\]
which is the Cartwright condition.
\end{proof}

\begin{corollary}[Hadamard product]\label{cor:hadamard}
Let $\Eone(w)\coloneqq(1-w)e^w$. Then
\[
    \Gf(z)
    =
    e^{\gamma z}
    \prod_{\rho\in Z(\Gf)}\Eone\left(\frac{z}{\rho}\right).
\]
The product converges locally uniformly and is independent of the
ordering of the zeros.
\end{corollary}

\begin{proof}
By \eqref{eq:zero-counting-asymptotic}, the exponent of convergence of
the zeros is $1$. Hadamard's factorization theorem
\cite[Ch.~5, \S3.2, Thm.~8, pp.~208--212]{Ahlfors} therefore implies
\[
    \Gf(z)
    =
    e^{az+b}
    \prod_{\rho\in Z(\Gf)}\Eone\left(\frac{z}{\rho}\right).
\]
Since $\Gf(0)=1$, we may take $b=0$. The logarithmic derivative of
every primary factor vanishes at the origin, so
\[
    a=\frac{\Gf'(0)}{\Gf(0)}=\gamma.
\]
Normal convergence follows from $\sum_{\rho}\abs{\rho}^{-2}<\infty$.
\end{proof}

\section{A spectral model}\label{sec:spectral}

The Markov transform \eqref{eq:markov} also leads to a direct spectral model. The positive
zeros appear as eigenvalues of a codimension-one compression of the
diagonal operator with eigenvalues $1/n$.

Consider $\ell^2(\NN)$ with its standard orthonormal basis $(e_n)_{n\geq1}$.
Define
\[
    Te_n\coloneqq\frac1n e_n
    \quad\text{and}\quad
    u\coloneqq \sum_{n=1}^{\infty}\sqrt{\abs{G_n}}\,e_n.
\]
The vector $u$ is a unit vector by \eqref{eq:abs-Gn-sums-to-1}. Let
\[
    Q\coloneqq u\otimes u,
    \qquad
    P\coloneqq I-Q,
    \qquad
    \Tperp\coloneqq PTP.
\]
Thus $\Tperp$ is the \defi{compression} of $T$ to $u^\perp$, extended by zero
on $\CC u$.

\begin{proposition}[Spectral compression model]\label{prop:resolvent}
For $z\notin\NN$, we have
\begin{equation}\label{eq:resolvent-identity}
    \Hf(z)
    =
    \left\langle(I-zT)^{-1}u,u\right\rangle.
\end{equation}
The nonzero spectrum of $\Tperp$ is
\[
    \sigma(\Tperp)\setminus\{0\}
    =
    \left\{\frac1{\rho_n}:n\geq1\right\},
\]
and every nonzero eigenvalue is simple.
\end{proposition}

\begin{proof}
The resolvent identity \eqref{eq:resolvent-identity} follows at once from \eqref{eq:markov}:
\[
    \left\langle(I-zT)^{-1}u,u\right\rangle
    =
    \sum_{n=1}^{\infty}\frac{\abs{G_n}}{1-z/n}
    =
    \Hf(z).
\]

Suppose that $\Tperp x=\lambda x$, where $\lambda\neq0$. Then $x\perp u$
and
\[
    (T-\lambda I)x=cu
\]
for some scalar $c$. If
$\lambda\notin\{1,1/2,1/3,\ldots\}$, then
\[
    x=c(T-\lambda I)^{-1}u.
\]
The condition $x\perp u$ is equivalent to
\[
    \left\langle(T-\lambda I)^{-1}u,u\right\rangle=0,
\]
or $\Hf(1/\lambda)=0$.

No number $1/n$ is an eigenvalue of $\Tperp$. Indeed, the $n$th coordinate
of $(T-n^{-1}I)x=cu$ forces $c=0$, after which $x$ is a multiple of
$e_n$; but $e_n\notin u^\perp$. Conversely, every zero of
$\Hf(1/\lambda)$ produces an eigenvector. \Cref{thm:zeros}
completes the proof.
\end{proof}

Thus
\[
    \frac1n>\frac1{\rho_n}>\frac1{n+1}
\]
is the eigenvalue interlacing for this codimension-one compression.

\section{Relative determinants and trace identities}\label{sec:determinants}

The compression differs from the original diagonal operator by a
finite-rank perturbation in the following relative form:
\[
    T-\Tperp=QT+TQ-QTQ.
\]
For $z\notin\NN$, we therefore have
\[
    (I-z\Tperp)(I-zT)^{-1}
    =I+z(T-\Tperp)(I-zT)^{-1},
\]
which differs from the identity by a finite-rank operator. Thus the
perturbation determinant in \Cref{def:relative-determinant} is defined.

\begin{definition}\label{def:relative-determinant}
The \defi{relative Fredholm determinant} of the pair $(\Tperp,T)$ is
\[
    D(z)
    \coloneqq 
    \det\left((I-z\Tperp)(I-zT)^{-1}\right).
\]
It is relative because it compares the resolvents of the compression
$\Tperp$ and the original diagonal operator $T$, rather than forming an
absolute determinant of either infinite-dimensional operator.
\end{definition}

\begin{theorem}[Relative determinant formula]\label{thm:determinant}
We have
\begin{equation}\label{eq:relative-determinant-identity}
    D(z)=\Hf(z)
\end{equation}
and
\begin{equation}\label{eq:relative-determinant-product}
    \Hf(z)
    =
    \prod_{n=1}^{\infty}
    \frac{1-z/\rho_n}{1-z/n}.
\end{equation}
The product converges locally uniformly away from its poles.
\end{theorem}

\begin{proof}
Decompose
\[
    \ell^2(\NN)=\CC u\oplus u^\perp
\]
and write
\[
    T=
    \begin{pmatrix}
        m&v^*\\
        v&T_0
    \end{pmatrix}
    \quad\text{and}\quad
    \Tperp=
    \begin{pmatrix}
        0&0\\
        0&T_0
    \end{pmatrix}.
\]
Here $m\coloneqq\langle Tu,u\rangle$. For $\abs{z}$ sufficiently small, put
$R_0\coloneqq(I-zT_0)^{-1}$ and
\[
    S(z)\coloneqq 1-zm-z^2\langle R_0v,v\rangle.
\]
The Schur-complement formula gives
\[
    \left\langle(I-zT)^{-1}u,u\right\rangle=S(z)^{-1}.
\]
Moreover, we have
\[
    I-zT
    =
    \begin{pmatrix}
        1&-zv^*R_0\\
        0&I
    \end{pmatrix}
    \begin{pmatrix}
        S(z)&0\\
        -zv&I-zT_0
    \end{pmatrix}.
\]
It follows that
\begin{equation}\label{eq:relative-determinant-factorization}
\begin{aligned}
    (I-z\Tperp)(I-zT)^{-1}
    ={}&
    \begin{pmatrix}
        1&0\\
        zv&I
    \end{pmatrix}
    \begin{pmatrix}
        S(z)^{-1}&zS(z)^{-1}v^*R_0\\
        0&I
    \end{pmatrix}.
\end{aligned}
\end{equation}
Both factors on the right-hand side of
\eqref{eq:relative-determinant-factorization} differ from the identity
by finite-rank operators. The first is unipotent and has Fredholm
determinant $1$; the determinant of the second factor is $S(z)^{-1}$. Hence
\begin{equation}\label{eq:relative-determinant-local}
    D(z)=S(z)^{-1}
    =
    \left\langle(I-zT)^{-1}u,u\right\rangle
    =
    \Hf(z)
\end{equation}
for $\abs{z}$ small. Both sides of \eqref{eq:relative-determinant-local}
are meromorphic on
$\CC\setminus\NN$, so the equality $D(z)=\Hf(z)$ holds there.

Put
\[
    P(z)
    \coloneqq
    \prod_{n=1}^{\infty}
    \frac{1-z/\rho_n}{1-z/n}.
\]
The interlacing inequalities imply
\[
    0<\frac1n-\frac1{\rho_n}<\frac1{n^2},
\]
so $P$ converges locally uniformly away from the positive integers.
Both $T$ and $\Tperp$ are Hilbert--Schmidt, while $T-\Tperp$ has
finite rank. The multiplicativity and eigenvalue-product formulas for
Fredholm and regularized determinants
\cite[Thm.~3.8, p.~256; Thm.~4.2, p.~258; Thm.~6.2,
p.~262]{Simon1977} therefore give
\begin{equation}\label{eq:relative-det2-quotient}
    D(z)
    =
    e^{z\Tr(T-\Tperp)}
    \frac{\det_2(I-z\Tperp)}{\det_2(I-zT)}
    =e^{cz}P(z)
\end{equation}
for some real constant $c$.

Only the normalization $c=0$ is specific to the present pair. For
$r\geq1$, the Markov representation gives
\[
    \frac1{2(1+r)}\leq D(-r)<1.
\]
Writing the $n$th factor of $P(-r)$ as $1-\delta_n(r)$, where
\[
    \delta_n(r)
    \coloneqq\frac{r(\rho_n-n)}{\rho_n(n+r)},
    \qquad
    0<\delta_n(r)\leq\frac{r}{n(n+r)},
\]
shows that $\log P(-r)=O(\log(1+r))$: the first factor is bounded
below by $1/\rho_1$, and for $n\geq2$ we have
$-\log(1-\delta_n)\leq2\delta_n$. Thus
$\log(D(-r)/P(-r))=O(\log(1+r))$. Equation
\eqref{eq:relative-det2-quotient} now forces $c=0$, proving
\eqref{eq:relative-determinant-product}.
\end{proof}

\begin{corollary}[Paired product]\label{cor:paired-product}
We have
\begin{equation}\label{eq:paired-product}
    \Gf(z)
    =
    \prod_{n=1}^{\infty}
    \left(1+\frac{z}{n}\right)
    \left(1-\frac{z}{\rho_n}\right).
\end{equation}
The product converges locally uniformly on $\CC$.
\end{corollary}

\begin{proof}
Multiply \eqref{eq:relative-determinant-product} by Euler's product
\[
    \frac{\sin\pi z}{\pi z}
    =
    \prod_{n=1}^{\infty}
    \left(1-\frac{z}{n}\right)
    \left(1+\frac{z}{n}\right).
\]
The paired factors satisfy
\[
    \left(1+\frac{z}{n}\right)
    \left(1-\frac{z}{\rho_n}\right)
    =
    1+O_K(n^{-2})
\]
uniformly for $z$ in a fixed compact set $K$. Hence the paired product
converges locally uniformly.
\end{proof}

\begin{corollary}[Laguerre--P\'olya property]\label{cor:LP}
The Gregory function belongs to the Laguerre--P\'olya class, the
locally uniform closure of real polynomials with only real zeros.
Consequently, every derivative of $\Gf$ has only real zeros.
\end{corollary}

\begin{proof}
The finite partial products in \eqref{eq:paired-product} are
real-rooted polynomials and converge locally uniformly to $\Gf$.
The classical Laguerre--P\'olya characterization therefore places
$\Gf$ in the class
\cite[Def.~3.1 and Thms.~3.2--3.3, pp.~42--46]{HirschmanWidder}.
Cauchy's integral formula, Rolle's theorem, and Hurwitz's theorem then
show that every derivative of $\Gf$ has only real zeros.
\end{proof}

For $k\geq1$, put
\[
    \mu_k
    \coloneqq 
    \langle T^ku,u\rangle
    =
    \sum_{n=1}^{\infty}\frac{\abs{G_n}}{n^k}.
\]
Then
\[
    \Hf(z)=1+\sum_{k=1}^{\infty}\mu_kz^k
    \quad\text{for}\quad \abs{z} < 1.
\]

Recall that an operator $A$ on a Hilbert space is trace class if its
singular values are summable, equivalently if $\Tr\abs{A}<\infty$.

\begin{theorem}[Relative trace identities]\label{thm:trace}
For every positive integer $k$, the operator $T^k-\Tperp^k$ is trace class
and
\begin{equation}\label{eq:trace-identity}
    \tau_k
    \coloneqq 
    \Tr(T^k-\Tperp^k)
    =
    \sum_{n=1}^{\infty}
    \left(\frac1{n^k}-\frac1{\rho_n^k}\right).
\end{equation}
Moreover, we have
\[
    \log\Hf(z)
    =
    \sum_{k=1}^{\infty}\frac{\tau_k}{k}z^k
    \quad\text{for}\quad \abs{z} < 1.
\]
In particular, we have
\[
    \gamma
    =
    \Tr(T-\Tperp)
    =
    \sum_{n=1}^{\infty}
    \left(\frac1n-\frac1{\rho_n}\right).
\]
\end{theorem}

\begin{proof}
The telescoping identity
\[
    T^k-\Tperp^k
    =
    \sum_{j=0}^{k-1}T^{k-1-j}(T-\Tperp)\Tperp^j
\]
shows that $T^k-\Tperp^k$ is trace class. Since
\(\lVert T\rVert,\lVert\Tperp\rVert\leq1\), the same identity gives
\[
    \lVert T^k-\Tperp^k\rVert_1
    \leq
    k\lVert T-\Tperp\rVert_1.
\]
Thus the resolvent expansion below converges in trace norm for
\(\abs{z}<1\). The logarithmic-derivative formula for Fredholm determinants
\cite[Eq.~(7.9), p.~268]{Simon1977} gives
\begin{equation}\label{eq:relative-log-derivative}
\begin{aligned}
    \frac{D'(z)}{D(z)}
    &=
    \Tr\left(
        T(I-zT)^{-1}-\Tperp(I-z\Tperp)^{-1}
    \right)\\
    &=
    \sum_{k=1}^{\infty}
    z^{k-1}\Tr(T^k-\Tperp^k).
\end{aligned}
\end{equation}
Integrating from $0$, and using $D(0)=1$, we obtain
\[
    \log D(z)
    =
    \sum_{k=1}^{\infty}
    \frac{z^k}{k}\Tr(T^k-\Tperp^k).
\]
On the other hand, \eqref{eq:relative-determinant-product} also yields
\[
    \log D(z)
    =
    \sum_{k=1}^{\infty}
    \frac{z^k}{k}
    \sum_{n=1}^{\infty}
    \left(\frac1{n^k}-\frac1{\rho_n^k}\right).
\]
For $k=1$, the inner series converges absolutely because its terms are
$O(n^{-2})$; for $k\geq2$, convergence is immediate. Comparison of
coefficients proves the trace identities.

Finally, we have
\[
    T-\Tperp=QT+TQ-QTQ,
\]
so
\[
    \Tr(T-\Tperp)=\langle Tu,u\rangle=\mu_1.
\]
Since $\Hf'(0)=\Gf'(0)=\gamma$, we have $\mu_1=\gamma$.
\end{proof}

For reference, the first few trace polynomials are
\[
\begin{aligned}
    \tau_1&=\mu_1,\\
    \tau_2&=2\mu_2-\mu_1^2,\\
    \tau_3&=3\mu_3-3\mu_1\mu_2+\mu_1^3,\\
    \tau_4&=
    4\mu_4-4\mu_1\mu_3-2\mu_2^2
    +4\mu_1^2\mu_2-\mu_1^4.
\end{aligned}
\]

\section{The reciprocal transform}\label{sec:reciprocal}

The \defi{reciprocal transform} $1/\Hf$ is governed by the same
real-zero structure. Its residues at the positive zeros produce a second
expansion for Euler's constant.

\begin{theorem}[Reciprocal partial-fraction expansion]\label{thm:reciprocal}
Let
\[
    \omega_n\coloneqq \frac1{\Hf'(\rho_n)}>0.
\]
Then
\begin{equation}\label{eq:reciprocal-transform-expansion}
    \frac1{\Hf(z)}
    =
    1-
    z\sum_{n=1}^{\infty}
    \frac{\omega_n}{\rho_n(\rho_n-z)}.
\end{equation}
The series converges locally uniformly away from the $\rho_n$.
Consequently, we have
\[
    \gamma
    =
    \sum_{n=1}^{\infty}
    \frac1{\rho_n^2\Hf'(\rho_n)}.
\]
\end{theorem}

\begin{proof}
By \Cref{cor:pick}, the function $-1/\Hf$ is a Pick
function. Its poles are the simple real poles $\rho_n$, and near such
a pole
\[
    -\frac1{\Hf(z)}
    =
    \frac{\omega_n}{\rho_n-z}+O(1).
\]
The Herglotz representation theorem and its Stieltjes inversion formula
\cite[Thm.~B.2, pp.~300--301]{Teschl} provide a positive representing
measure on the real line. On every compact interval avoiding the
$\rho_n$, the function $-1/\Hf$ extends holomorphically across the real
axis and is real-valued there. Its imaginary part therefore tends to zero
uniformly as the axis is approached, and Stieltjes inversion shows that
the representing measure vanishes on that interval. Finally,
\cite[Lem.~B.10, p.~305]{Teschl} identifies the mass at $\rho_n$ as
\[
    \lim_{\varepsilon\downarrow0}
    \frac{\varepsilon}{i}
    \left(-\frac1{\Hf(\rho_n+i\varepsilon)}\right)
    =\omega_n.
\]
Thus the measure consists precisely of these atoms, and hence
\[
    -\frac1{\Hf(z)}
    =
    \alpha z+\beta
    +
    \sum_{n=1}^{\infty}
    \omega_n
    \left(
        \frac1{\rho_n-z}
        -\frac{\rho_n}{1+\rho_n^2}
    \right),
\]
where $\alpha\geq0$ and
$\sum_n\omega_n/(1+\rho_n^2)<\infty$.
Equations \eqref{eq:sector-upper} and \eqref{eq:completed-reflection}
imply
\[
    -\frac1{\Hf(iy)}=O(\log y)
    \quad\text{as}\quad y \to \infty,
\]
so
\begin{equation}\label{eq:herglotz-alpha-limit}
    \lim_{y\to\infty}\frac{-1/\Hf(iy)}{iy}=0.
\end{equation}
In the Herglotz representation, the limit in \eqref{eq:herglotz-alpha-limit}
is $\alpha$, and hence $\alpha=0$. Subtracting the value at $z=0$, and using
$\Hf(0)=1$, we obtain
\[
    -\frac1{\Hf(z)}+1
    =
    \sum_{n=1}^{\infty}
    \omega_n
    \left(
        \frac1{\rho_n-z}-\frac1{\rho_n}
    \right).
\]
The expansion \eqref{eq:reciprocal-transform-expansion} follows. Its
local uniform
convergence away from the poles follows from
$\sum_n\omega_n/(1+\rho_n^2)<\infty$ and the fact that
$\inf_n\rho_n>0$. Differentiating at the origin, we find
\[
    -\gamma
    =
    \left(\frac1{\Hf}\right)'(0)
    =
    -\sum_{n=1}^{\infty}\frac{\omega_n}{\rho_n^2}.
\]
\end{proof}

\section{The relative zero zeta function}\label{sec:relative-zeta}

The trace identities can be packaged as a zeta function. The natural
object compares the natural-number sequence
$(n)_{n\geq1}$ with the positive-zero sequence $(\rho_n)_{n\geq1}$.

\begin{definition}\label{def:relative-zeta}
For $\re s>0$, define the \defi{relative zero zeta function}
\begin{equation}\label{eq:relative-zero-zeta}
    \XiG(s)
    \coloneqq 
    \sum_{n=1}^{\infty}
    \left(n^{-s}-\rho_n^{-s}\right).
\end{equation}
It is relative in the sense of \Cref{def:relative-determinant}: it measures
the displacement from the reference sequence $(n)_{n\geq1}$ of positive
integers to the positive-zero sequence $(\rho_n)_{n\geq1}$.
\end{definition}

\begin{proposition}\label{prop:Xi}
The series in \eqref{eq:relative-zero-zeta} converges locally uniformly
in $\re s>0$.
Thus $\XiG$ is holomorphic there. For every positive integer $k$, we have
\[
    \XiG(k)=\tau_k.
\]
In particular, we have
\[
    \XiG(1)=\gamma.
\]
\end{proposition}

\begin{proof}
If $K$ is a compact subset of $\re s>0$, put
\[
    \sigma_0\coloneqq \inf_{s\in K}\re s>0.
\]
For $s\in K$, we have
\[
    n^{-s}-\rho_n^{-s}
    =
    s\int_n^{\rho_n}x^{-s-1}\dd x,
\]
so
\[
    \abs{n^{-s}-\rho_n^{-s}}
    \leq
    C_Kn^{-\sigma_0-1}.
\]
The Weierstrass test proves local uniform convergence. The values at
positive integers follow from \eqref{eq:trace-identity}.
\end{proof}

To understand $\XiG(s)$ as $s\to0^+$, we will represent it as a Mellin
transform of $\log(1/\Hf(-t))$. The next two lemmas identify the
large-$t$ behavior that controls the origin.

\begin{lemma}\label{lem:gregory-tail-estimates}
As $t\to\infty$, we have
\begin{equation}\label{eq:gregory-tail-mass}
    \sum_{n>t}\abs{G_n}\sim\frac1{\log t},
\end{equation}
\begin{equation}\label{eq:gregory-truncated-first-moment}
    \sum_{n\leq t}n\abs{G_n}=O\left(\frac{t}{\log^2t}\right),
\end{equation}
and
\begin{equation}\label{eq:gregory-tail-remainder}
    t\sum_{n>t}\frac{\abs{G_n}}{n+t}=O(\log^{-2}t).
\end{equation}
\end{lemma}

\begin{proof}
The estimate \eqref{eq:gregory-tail-mass} follows from
\eqref{eq:gregory-coefficient-asymptotic} by partial summation,
equivalently by
comparison with
\[
    \int_t^\infty\frac{\dd x}{x\log^2x}=\frac1{\log t}.
\]
The estimate \eqref{eq:gregory-truncated-first-moment} follows from
\[
    \sum_{n\leq t}n\abs{G_n}\ll\sum_{3\leq n\leq t}\frac1{\log^2n}
    =O\left(\frac{t}{\log^2t}\right),
\]
with the finitely many initial terms absorbed into the constant. For
\eqref{eq:gregory-tail-remainder}, use
\[
    t\sum_{n>t}\frac{\abs{G_n}}{n+t}
    \leq
    t\sum_{n>t}\frac{\abs{G_n}}{n}
    \ll
    t\int_t^\infty\frac{\dd x}{x^2\log^2x}
    =O(\log^{-2}t).
\]
\end{proof}

\begin{lemma}\label{lem:H-negative}
As $t\to\infty$, we have
\begin{equation}\label{eq:H-negative-asymptotic}
    \Hf(-t)\sim\frac1{\log t}.
\end{equation}
\end{lemma}

\begin{proof}
Equation \eqref{eq:markov} shows
\[
    \Hf(-t)
    =
    \sum_{n=1}^{\infty}\abs{G_n}\frac{n}{n+t}.
\]
Equation \eqref{eq:gregory-truncated-first-moment} implies
\[
    \sum_{n\leq t}\abs{G_n}\frac{n}{n+t}
    \leq
    \frac1t\sum_{n\leq t}n\abs{G_n}
    =
    O(\log^{-2}t),
\]
and \eqref{eq:gregory-tail-mass} together with
\eqref{eq:gregory-tail-remainder} imply
\[
    \sum_{n>t}\abs{G_n}\frac{n}{n+t}
    =
    \sum_{n>t}\abs{G_n}
    -
    t\sum_{n>t}\frac{\abs{G_n}}{n+t}
    =
    \frac1{\log t}+O(\log^{-2}t).
\]
The assertion follows.
\end{proof}

\begin{theorem}[Mellin representation and origin singularity]\label{thm:mellin}
For $0<\re s<1$, we have
\begin{equation}\label{eq:relative-zeta-mellin}
    \XiG(s)
    =
    \frac{s\sin\pi s}{\pi}
    \int_0^\infty
    t^{-s-1}\log\frac1{\Hf(-t)}
    \dd t.
\end{equation}
As $s\to0^+$, we have
\begin{equation}\label{eq:relative-zeta-origin-asymptotic}
    \XiG(s)
    =
    -s\log s-\gamma s+o(s).
\end{equation}
In particular, $\XiG$ has no meromorphic continuation through $s=0$.
\end{theorem}

\begin{proof}
Taking logarithms in \eqref{eq:relative-determinant-product} gives
\[
    \log\frac1{\Hf(-t)}
    =
    \sum_{n=1}^{\infty}
    \left[
        \log\left(1+\frac{t}{n}\right)
        -
        \log\left(1+\frac{t}{\rho_n}\right)
    \right].
\]
Every summand is positive. For real $s\in(0,1)$, Tonelli's theorem and
\begin{equation}\label{eq:log-mellin-kernel}
    \int_0^\infty
    t^{-s-1}\log\left(1+\frac{t}{a}\right)\dd t
    =
    \frac{\pi a^{-s}}{s\sin\pi s}
\end{equation}
combine to prove \eqref{eq:relative-zeta-mellin}. Near the origin,
\(\log(1/\Hf(-t))=O(t)\); at infinity, \Cref{lem:H-negative} gives
\(\log(1/\Hf(-t))=O(\log\log t)\). The integral in
\eqref{eq:relative-zeta-mellin} therefore converges locally uniformly for
\(0<\re s<1\) and defines a holomorphic function there. The identity
extends throughout the strip.

Equation \eqref{eq:H-negative-asymptotic} shows
\[
    \log\frac1{\Hf(-t)}
    =
    \log\log t+o(1).
\]
Moreover, we have
\[
\begin{aligned}
    \int_1^\infty t^{-s-1}\log\log t\dd t
    &=
    \int_0^\infty e^{-su}\log u\dd u\\
    &=
    \frac{-\gamma-\log s}{s}.
\end{aligned}
\]
To control the error, write it as $r(t)=o(1)$. Given $\eta>0$,
choose $T$ so that $\abs{r(t)}\leq\eta$ for $t\geq T$. Then
\[
    \int_T^\infty t^{-s-1}\abs{r(t)}\dd t
    \leq \frac{\eta T^{-s}}s,
\]
whereas the integral over $[1,T]$ is $O_T(1)$. Thus the error
contributes $o(s^{-1})$. The integral over $(0,1)$ remains bounded
because $\log(1/\Hf(-t))=O(t)$ at the origin. Since
\[
    \frac{s\sin\pi s}{\pi}=s^2+O(s^4),
\]
the asymptotic \eqref{eq:relative-zeta-origin-asymptotic} follows.

If $\XiG$ had a meromorphic continuation through $0$, its limit along
the positive real axis would force the singularity to be removable.
The quotient $\XiG(s)/s$ would then have a finite limit, contrary to
\eqref{eq:relative-zeta-origin-asymptotic}.
\end{proof}

\section{Regularized displacement and a computational reduction}\label{sec:regularized}

We finish with a monotonicity question suggested by the asymptotic
expansion. The material in this section is separate from \Cref{thm:main}.

\begin{definition}\label{def:regularized-displacement}
For the positive-zero displacement sequence $(\varepsilon_n)_{n\geq1}$ of
\Cref{def:displacement}, set
\[
    c_n\coloneqq \frac1{\varepsilon_n}-\log n.
\]
We call $c_n$ the $n$th \defi{regularized displacement}. The
subtraction removes the leading logarithmic growth of $1/\varepsilon_n$.
\end{definition}
Equation \eqref{eq:displacement-asymptotic} implies
\begin{equation}\label{eq:regularized-displacement-asymptotic}
    c_n
    =
    \gamma
    +
    \frac{6\zeta(3)}{(\log n)^2}
    -
    \frac{12\gamma\zeta(3)+7\pi^4/90}{(\log n)^3}
    +
    O((\log n)^{-4}).
\end{equation}
In particular, $c_n\to\gamma$. Eventual monotonicity requires more
than \eqref{eq:regularized-displacement-asymptotic}, since its remainder
does not by itself control successive differences.

Recall from \Cref{lem:theta-decreasing} that $\theta(t)\in(0,\pi)$
denotes the continuous argument of $\Phi(t)$ for $t>1$. At a positive
zero, we have
\[
    \theta(\rho_n)=\pi\varepsilon_n.
\]
Define the \defi{continuous regularized displacement}
\[
    c(t)
    \coloneqq 
    \frac{\pi}{\theta(t)}
    -
    \log\left(t-\frac{\theta(t)}{\pi}\right).
\]
Then $c(\rho_n)=c_n$.

\begin{proposition}[Eventual monotonicity]\label{prop:eventual}
The sequence $(c_n)_{n\geq1}$ is strictly decreasing for all sufficiently
large $n$.
\end{proposition}

\begin{proof}
Put
\[
    p(t)\coloneqq \frac{\theta(t)}{\pi}
    \quad\text{and}\quad
    L\coloneqq\log t.
\]
Write
\[
    R(t,x)\coloneqq\frac{\Gamma(t-x)}{\Gamma(t+1)}.
\]
The uniform gamma-ratio expansion
\eqref{eq:gamma-ratio-positive-axis} and the uniform digamma expansion
\cite[Eq.~(5.11.2)]{NIST} imply the estimates below. Indeed,
substituting $t-x$ and $t+1$ into the digamma expansion and expanding
the elementary terms uniformly for $0\leq x\leq1$ gives
\[
    \psi(t-x)-\psi(t+1)
    =-\frac{x+1}{t}+O(t^{-2})
    \quad\text{for}\quad 0 \leq x \leq 1.
\]
and hence
\[
\begin{aligned}
    t\frac{\partial R}{\partial t}(t,x)
    &=tR(t,x)\bigl(\psi(t-x)-\psi(t+1)\bigr)\\
    &=-(x+1)t^{-x-1}+O(t^{-x-2}),
\end{aligned}
\]
uniformly for $0\leq x\leq1$.
With $I(L)$ as in the proof of \Cref{thm:asymptotic}, integration yields
\begin{align}
    \Phi(t)
    &=\frac{I(L)}{t}+O\left(\frac1{t^2L}\right),
    \label{eq:phase-and-derivative-first}\\
    t\Phi'(t)
    &=\frac{I'(L)-I(L)}{t}
    +O\left(\frac1{t^2L}\right).
    \label{eq:phase-and-derivative-second}
\end{align}
Since $I(L)\sim L^{-1}$, dividing
\eqref{eq:phase-and-derivative-second} by
\eqref{eq:phase-and-derivative-first}, we find
\begin{equation}\label{eq:phase-derivative-logarithmic}
    t\theta'(t)
    =
    \im\frac{I'(L)}{I(L)}+O(t^{-1}).
\end{equation}
Thus \Cref{lem:watson-differentiated} justifies differentiating the
phase expansion with respect to $L$. Differentiating
\eqref{eq:phase-asymptotic}, we obtain
\begin{align*}
    p(t)
    ={}&
    \frac1L-\frac{\gamma}{L^2}+\frac{\gamma^2}{L^3}
    -\frac{\gamma^3+6\zeta(3)}{L^4}+O(L^{-5}),\\
    tp'(t)
    ={}&
    -\frac1{L^2}+\frac{2\gamma}{L^3}
    -\frac{3\gamma^2}{L^4}
    +\frac{4(\gamma^3+6\zeta(3))}{L^5}
    +O(L^{-6}).
\end{align*}
A direct division now produces
\begin{equation}\label{eq:phase-quotient-derivative}
    -\frac{tp'(t)}{p(t)^2}
    =
    1-\frac{12\zeta(3)}{L^3}+O(L^{-4}).
\end{equation}
On the other hand, $p(t)=O(L^{-1})$ and
$p'(t)=O(t^{-1}L^{-2})$, so
\begin{equation}\label{eq:logarithmic-correction-derivative}
    \frac{t(1-p'(t))}{t-p(t)}
    =1+O\left(\frac1{tL}\right).
\end{equation}
Since
\[
    c'(t)
    =
    -\frac{p'(t)}{p(t)^2}
    -\frac{1-p'(t)}{t-p(t)},
\]
subtracting \eqref{eq:logarithmic-correction-derivative} from
\eqref{eq:phase-quotient-derivative}, we obtain
\begin{equation}\label{eq:continuous-regularized-derivative-asymptotic}
    c'(t)
    =
    -\frac{12\zeta(3)}{t(\log t)^3}
    +
    O\left(\frac1{t(\log t)^4}\right).
\end{equation}
Therefore $c'(t)<0$ for all sufficiently large $t$. Since $\rho_n$ is
increasing, $c_n=c(\rho_n)$ is eventually strictly decreasing.
\end{proof}

The asymptotic formula and \Cref{prop:eventual} suggest that the
regularization captures a global monotone approach to $\gamma$, which
would strengthen the limiting relation into a uniform estimate for every
positive zero.

\begin{conjecture}\label{conj:displacement}
The sequence $(c_n)_{n\geq1}$, given by
\[
    c_n=\frac1{\rho_n-n}-\log n,
\]
is strictly decreasing.
\end{conjecture}

Since $c_n\to\gamma$, \Cref{conj:displacement} would imply
\[
    \rho_n-n<\frac1{\log n+\gamma}
    \quad\text{for}\quad n \geq 1.
\]

The conjecture reduces to the positivity of one real-analytic
function. The beta integral gives
\begin{equation}\label{eq:phase-beta-integral}
    \Phi(t)
    =
    \int_0^1
    \frac{u^{t-2}}
         {\log\frac{u}{1-u}-i\pi}
    \dd u.
\end{equation}
At an integer $n\geq2$, \eqref{eq:phase} specializes to
$B(n)=\pi\abs{G_n}$. Taking imaginary parts in
\eqref{eq:phase-beta-integral}
therefore implies
\begin{equation}\label{eq:schroeder-unit}
    \abs{G_n}
    =
    \int_0^1
    \frac{u^{n-2}}
         {\left(\log\frac{u}{1-u}\right)^2+\pi^2}
    \dd u.
\end{equation}
The substitution $v=(1-u)/u$ recovers Schr\"oder's representation
\eqref{eq:schroeder}. Applying the inverse substitution to
\eqref{eq:schroeder} and reindexing produces the Hausdorff-moment formula
\begin{equation}\label{eq:gregory-hausdorff}
    \abs{G_{n+1}}
    =
    \int_0^1 u^n\,
    \frac{\dd u}
         {u\left(\left(\log\frac{u}{1-u}\right)^2+\pi^2\right)}
    \quad\text{for}\quad n \geq 0.
\end{equation}
Consequently, for every $n,k\geq0$, we have
\[
    (-1)^k\Delta^k\abs{G_{n+1}}
    =
    \int_0^1 u^n(1-u)^k\,
    \frac{\dd u}
         {u\left(\left(\log\frac{u}{1-u}\right)^2+\pi^2\right)}
    \geq0.
\]
This is the finite-difference complete monotonicity proved by Qi and
Zhang \cite[Thm.~2, p.~988]{QiZhang}; the moment formula \eqref{eq:gregory-hausdorff}
also explains why it is natural
in the present phase-integral setting.

Set
\[
    s\coloneqq t-1,
    \qquad
    L\coloneqq\log s,
    \qquad
    W(L)\coloneqq s\Phi(s+1),
    \qquad
    \Theta(L)\coloneqq \arg W(L).
\]
For real $L$, the imaginary part of $W(L)$ is positive, so
\[
    0<\Theta(L)<\pi.
\]
Define
\[
\begin{aligned}
    \Dcomp(L)
    \coloneqq {}&
    \left(\frac{\Theta(L)}{\sin\Theta(L)}\right)^2\\
    &+
    \left(1+e^{-L}\right)
    \pi\csc^2\Theta(L)
    \im\frac{W'(L)}{W(L)}.
\end{aligned}
\]

Our final proposition makes the remaining obstacle explicit; one
positivity inequality would promote the eventual monotonicity of
\Cref{prop:eventual} to the global monotonicity asserted in
\Cref{conj:displacement}.

\begin{proposition}[Positivity criterion for global monotonicity]\label{prop:computational}
If
\begin{equation}\label{eq:computational-positivity}
    \Dcomp(L)>0
    \quad\text{for}\quad L \geq -\log 2,
\end{equation}
then \Cref{conj:displacement} holds.
\end{proposition}

\begin{proof}
Put
\[
    p(t)\coloneqq \frac{\theta(t)}{\pi}.
\]
Then
\[
    c'(t)
    =
    -\frac{p'(t)}{p(t)^2}
    -
    \frac{1-p'(t)}{t-p(t)}.
\]
By \eqref{eq:theta-decreasing}, $p'(t)<0$. The stronger inequality
\begin{equation}\label{eq:p-derivative-criterion}
    -tp'(t)<p(t)^2
\end{equation}
implies $c'(t)<0$. To see this, put \(q\coloneqq-p'(t)>0\). Then
\[
    \frac{q}{p(t)^2}<\frac1t,
\]
while \(0<p(t)<1\) gives
\[
    \frac{1+q}{t-p(t)}>\frac1t.
\]
The negative term in $c'(t)$ therefore dominates the positive one.

Now $\Theta(L)=\theta(t)$ and
\[
    \Theta'(L)=(t-1)\theta'(t).
\]
Since $1+e^{-L}=t/(t-1)$, \eqref{eq:computational-positivity} is
equivalent to
\[
    -t\theta'(t)<\frac{\theta(t)^2}{\pi},
\]
which is \eqref{eq:p-derivative-criterion}. Thus $c'(t)<0$ whenever
$\Dcomp(\log(t-1))>0$.

Equation \eqref{eq:first-zero-half-bound} implies
\[
    \rho_1>\frac32,
\]
so
\[
    \log(\rho_n-1)\geq\log(\rho_1-1)>-\log2
\]
for every $n$. Hence
\[
    c_n=c(\rho_n)>c(\rho_{n+1})=c_{n+1}.
\]
\end{proof}

For computation, put
\[
    r(x)\coloneqq \log\frac{e^x-1}{x}.
\]
The substitutions $u\coloneqq e^{-v/s}$ and $v\coloneqq e^y$ transform this into the exact fixed-domain
formula
\[
    W(L)
    =
    \int_{-\infty}^{\infty}
    \frac{e^{y-e^y}}
         {L-y-r(e^{y-L})-i\pi}
    \dd y
\]
and
\[
    W'(L)
    =
    -
    \int_{-\infty}^{\infty}
    \frac{
        e^{y-e^y}
        \left(1+e^{y-L}r'(e^{y-L})\right)
    }{
        \left(L-y-r(e^{y-L})-i\pi\right)^2
    }
    \dd y.
\]
A certified proof of \eqref{eq:computational-positivity} should include
the verifier, its complete
output, and the dependency versions. Floating-point evidence alone does
not prove the conjecture. The reduction is included to isolate the
remaining problem, not to supply evidence for it.

\section{Further directions}

The methods above suggest three analytic problems. First, consider the
\defi{shifted family}
\[
    \Gf(a,z)
    \coloneqq
    \int_0^1\binom{a+x}{z}\dd x,
\]
which satisfies
\[
    \Gf(a+1,z)-\Gf(a,z)=\Gf(a,z-1).
\]
The structural Markov problem for this family is distinct from the
questions pursued here. From the present viewpoint, the remaining
problem is to determine how the order, indicator, zero displacement,
and relative determinant vary with $a$.

Second, by the standard correspondence between compactly supported
measures of infinite support and Jacobi matrices
\cite[Thm.~2.13, pp.~43--44]{Teschl}, the measure
\[
    \sum_{n=1}^{\infty}\abs{G_n}\delta_{1/n}
\]
determines a Jacobi matrix. An explicit description of its recurrence
coefficients would give a tridiagonal spectral model for the positive
Gregory zeros and might make their displacement inequalities more
transparent.

Third, \eqref{eq:relative-zeta-origin-asymptotic} shows that the
relative zero zeta function has a logarithmic singularity at the origin
rather than a meromorphic one.
Its continuation and full singular expansion should therefore be
studied on a logarithmic covering of a punctured neighborhood of the
origin.

\enlargethispage{2\baselineskip}

\end{document}